\documentclass[a4paper,12pt]{article}
\usepackage{amscd}
\usepackage{amssymb}
\usepackage{amsfonts}
\usepackage{stmaryrd}
\usepackage{amsmath}
\allowdisplaybreaks
\usepackage{amsthm}
\usepackage{bbm}
\usepackage{CJK}
\usepackage{fancyhdr}
\usepackage{graphicx}

\usepackage{geometry}
\usepackage{ifpdf}
\ifpdf
  \usepackage[colorlinks=true,linkcolor=blue,citecolor=red, final,backref=page,hyperindex]{hyperref}
\else
  \usepackage[colorlinks,final,backref=page,hyperindex,hypertex]{hyperref}
\fi
\usepackage{tikz}
\usepackage[active]{srcltx}

\usepackage{pxfonts}
\usepackage{indentfirst}
\usepackage{latexsym}
\usepackage{mathrsfs}
\usepackage{xypic}
\usepackage{a4wide}
\usepackage{microtype}

\usepackage[sf,tiny,bf]{titlesec}
\titlelabel{\thetitle.~}

  \theoremstyle{plain}
    \newtheorem{thm}{Theorem}[section]
    \newtheorem{prop}[thm]{Proposition}

   \newtheorem{lemma}[thm]{Lemma}
    \newtheorem{corollary}[thm]{Corollary}

\theoremstyle{definition}
    \newtheorem{defn}[thm]{Definition}

    \newtheorem{exam}[thm]{Example}

\theoremstyle{remark}

\newcommand{\bmx}{\begin{pmatrix}}
\newcommand{\emx}{\end{pmatrix}}

\newcommand{\A}{\mathcal{A}}
\newcommand{\ot}{\otimes}
\renewcommand{\b}{\beta}

\def\a{\alpha}
\def\O{\mathcal{O}}

\newcommand{\what}{\widehat}

\date{}
\usepackage{amssymb}

\begin{document}
\title{ \sf Cohomology  and deformations of $\mathcal{O}$-operators on Hom-associative algebras}
\author{\bf T. Chtioui, S. Mabrouk, A. Makhlouf}
\author{{ Taoufik Chtioui$^{1}$
 \footnote {  E-mail: chtioui.taoufik@yahoo.fr}
, Sami Mabrouk$^{2}$
 \footnote { E-mail: mabrouksami00@yahoo.fr}
, Abdenacer Makhlouf$^{3}$
 \footnote {  E-mail: Abdenacer.Makhlouf@uha.fr. (Corresponding author)}
}\\
{\small 1.  University of Sfax, Faculty of Sciences Sfax,  BP
1171, 3038 Sfax, Tunisia} \\
{\small 2.  University of Gafsa, Faculty of Sciences Gafsa, 2112 Gafsa, Tunisia}\\
{\small 3.~ IRIMAS - D\'epartement de Math\'ematiques, 6, rue des fr\`eres Lumi\`ere,
F-68093 Mulhouse, France}}
 \maketitle

\begin{abstract}
In this paper, we introduce the cohomology theory of $\mathcal{O}$-operators on Hom-associative algebras. This cohomology can also be viewed  as the Hochschild cohomology of a certain Hom-associative algebra with coefficients in a suitable bimodule.
Next, we study infinitesimal and formal deformations of an $\O$-operator and show that they
are governed by the above-defined cohomology. Furthermore,   the notion of  Nijenhuis elements associated with an $\O$-operator  is introduced  to characterize trivial infinitesimal deformations.

\end{abstract}


{\bf keywords}: $\mathcal{O}$-operator, Maurer-Cartan equation, graded-Lie algebra, Hochschild cohomology, Formal deformation, Hom-associative {\bf r}-matrix.

\tableofcontents
\numberwithin{equation}{section}
\vspace{0.2cm}

\section{Introduction}
The concept of Rota-Baxter operators on associative algebras was introduced in the 1960th by G. Baxter  \cite{baxter} in the study of fluctuation theory in Probability and then the theory was established by G.-C. Rota \cite{rota} mainly  in connection with  Combinatorics.
Motivated by the study of associative analogues of Poisson structures, Uchino introduced the notion of  generalized Rota-Baxter operator (called also $\O$-operator) on a bimodule over an associative algebra in \cite{uchino1}.
In \cite{uchino2}, Uchino also introduced the notion of a strong Maurer-Cartan equation on a bimodule $(V;\ell,r)$. A solution of the strong Maurer-Cartan equation can be seen as the associative analogue of a closed $2$-form complementary to a Poisson structure.
Rota-Baxter operators and their generalization $\O$-operators are becoming rather popular due to their interest in various domain like nonassociative algebra and applied algebra. Their deformations and their controlling cohomologies have been studied for Lie algebras in \cite{tang} and  for associative algebras in \cite{def O-op on ass al}, see also \cite{CaiSheng} \cite{LazarevSheng}. 

The first instances of  Hom-type algebras appeared in $q$-deformations of algebras of vector fields in physics contexts. It turns out that the Jacobi condition is no longer satisfied in the $q$-deformations of Witt algebra and Virasoro algebra using Jackson derivative. Hartwig, Larsson and Silvestrov \cite{hls} introduced a so called Hom-Lie algebras, generalizing Lie algebras and  where the usual Jacobi identity is twisted by a homomorphism. The pending Hom-associative algebras were  introduced in \cite{ms}. Various classical structures were generalized within this framework. A huge research activity was dedicated to Hom-type algebras, due in part to the prospect of having a general framework in which one can produce many types of natural deformations  which are of interest to both mathematicians and physicists. Cohomology and deformations of Hom-Lie and Hom-associative algebras were studied in \cite{amm-ej-makh,hurleMakhlouf,sheng}. Notice  that  Cohomology  and deformations of $\mathcal{O}$-operators on Hom-Lie algebras were discussed in 
\cite{O-op on hom-lie}.

The purpose of this paper is to introduce and study a cohomology controlling deformations of $\mathcal{O}$-operators on Hom-associative algebras. Therefore, we consider one parameter formal deformations and provide relevant results about infinitesimal deformations and general $n$-order deformations. In addition, infinitesimal deformations are characterized using  Nijenhuis elements. In Section 2, relevant definitions and preliminary results are presented. They deal with Hom-associative algebras  and Hom-dendriform algebra, bimodule structures, Rota-Baxter structures, $\O$-operators, Nijenhuis operators, cohomology and graded Lie algebras.  Main results of this paper are given in Section 3 in which we provide  Cohomologies of $\O$-operators on Hom-associative algebras and  their deformations are considered in Section 4.

\section{Preliminaries}
In this section, we  recall the relevant definitions and the  notion of $\mathcal{O}$-operator on Hom-associative algebra.  We also recall Hom-dendriform structures induced from $\mathcal{O}$-operators. Finally, we construct a graded Lie algebra with explicit graded Lie bracket whose Maurer-Cartan elements are given by $\mathcal{O}$-operators.

\subsection{Hom-associative and Hom-dendriform algebras}
\begin{defn}
A Hom-associative algebra is a Hom-module  $(A, \alpha)$, consisting of a $\mathbb{K}$-vector space $A$ and a linear map $\alpha$,  together with a bilinear map $\mu : A \times A \rightarrow A,~ (a, b) \mapsto a\cdot  b$, that satisfies
\begin{align*}
\alpha (a) \cdot (b \cdot c) = (a\cdot  b) \cdot \alpha (c), ~~~ \text{ for all } ~ a, b, c \in A.
\end{align*}
\end{defn}

A Hom-associative algebra is called multiplicative if $\alpha (a \cdot b) = \alpha (a) \cdot \alpha (b)$. By a Hom-associative algebra, we shall always mean a multiplicative one, unless otherwise specified.
\begin{defn}
Let $(A,\mu,\a)$ be a Hom-associative algebra and $M$ a vector space. Let $\phi \in gl(M)$ and $l: A\otimes M \to M,\ r: M \otimes A \to M$  be two linear maps. We say that the tuple $(M,l,r,\phi)$ is a bimodule of $A$ if  for any $x,y  \in A $ and $v \in M$
\begin{align*}
 \phi l(x,v)=l(\a(x),\phi(v)),&\  \phi r(v,x)=r(\phi(v),\a(x)),\\
l(x y,\phi(v)) =&l(\a(x),l(y,v)), \\
 r(\phi(v),x y)=&r(r(v,x),\a(y)),\\
  l(\a(x),r(v,y))=&r(l(x,v),\a(y)).
\end{align*}
\end{defn}

Unless otherwise stated, $l(x,v)$ (resp. $r(x,v)$) may be denoted, depending on the situation, by $xv$ or $l(x)v$ (resp. $vx$ or $r(x)v$).

Recall that if we consider a Hom-associative algebra $(A, \mu,\a)$. Then  $(M,l,r,\phi)$ is a bimodule of $A$ if and only if  the vector space $A \oplus M$ (the split extension of $A$ by $M$) carries a Hom-associative structure with product given by
\begin{align*}
(x, u) \cdot_{l,r} (y, v ) =& ( x \cdot  y , l(x,u) + r(v,y) ),\ \forall x,y \in A, u,v \in M, \\
(\a,\phi)(x,u)=&(\a(x),\phi(u)).
\end{align*}
This is called the semi-direct product of $A$ with $M$.

In \cite{amm-ej-makh, makh-sil} the authors define a Hochschild-type cohomology of a Hom-associative algebra and study the one parameter formal deformation theory for these type of algebras. Let $(A, \alpha, \mu)$ be a Hom-associative algebra and $(M,l,r,\phi)$ be a bimodule. For each $n \geq 1$,  the group of $n$-cochains is defined as
\begin{align*}
C^n_{\alpha,\phi}(A, M):= \{ f : A^{\otimes n} \rightarrow M |~ \phi \circ f (a_1, \ldots, a_n ) = f (\alpha (a_1), \ldots, \alpha (a_n)) \}
\end{align*}
and the differential $\delta_{\alpha,\phi} : C^n_{\alpha, \phi} (A, M) \rightarrow C^{n+1}_{\alpha, \phi} (A, M)$ is given by
\begin{align}\label{hom-ass-diff}
(\delta_{\a,\phi} f)(a_1, \ldots, a_{n+1}) =~& \alpha^{n-1} (a_1)  f (a_2, \ldots, a_{n+1}) \\
~&+ \sum_{i=1}^n (-1)^{i} f (   \alpha (a_1), \ldots, \alpha (a_{i-1}), \mu(a_i, a_{i+1}), \alpha (a_{i+2}), \ldots, \alpha (a_{n+1}) ) \nonumber \\
~&+ (-1)^{n+1} f (a_1, \ldots, a_n)  \alpha^{n-1} (a_{n+1}). \nonumber
\end{align}
for $a_1, \ldots, a_{n+1} \in A$.
The cohomology of this complex is called the Hochschild cohomology of $A$ with coefficients in the bimodule $(M,l,r,\phi)$ and the cohomology groups are denoted by
$H^\bullet_{\a,\phi}(A,M)$.

In \cite{uchino1}, Uchino gave a characterization of an $\mathcal{O}$-operator on an associative algebra. This result is still true for Hom-associative algebras. Recall first the notion of an $\O$-operator on a Hom-associative algebra.

\begin{defn}
Let $(A,\mu,\a)$ be a Hom-associative algebra and $(M,l,r,\phi)$ be a bimodule. A linear map $T: M \to A$ is called an $\O$-operator with respect to $M$ if it satisfies
\begin{align*}
T\phi=\a T,\ \quad T(u)  T(v) = T ( l(T(u))v+r(T(v))u ), ~~~ \text{ for all } u, v \in M.
\end{align*}
\end{defn}

\begin{exam}
A Rota-Baxter operator $\mathcal{R}: A \to A$ of weight $0$ is an $\O$-operator on $A$ with respect to the adjoint representation $(A,L,R,\a)$.
\end{exam}

\begin{exam}\label{example1ass}
Let $\{x_1,x_2,x_3\}$  be a basis of a $3$-dimensional vector space
$A$ over $\mathbb{K}$. Consider the following multiplication $\mu$ and linear map
$\alpha$ on $A$ that define Hom-associative algebras over $\mathbb{K}^3${\rm :}
$$
\begin{array}{ll}
\begin{array}{lll}
 \mu( x_1,x_1)&=& a\ x_1, \ \\
 \mu(x_1,  x_2)&=&\mu( x_2 ,  x_1)=a\ x_2,\\
\mu(x_1 ,  x_3) &=& \mu(x_3,  x_1)=b\ x_3,\\
 \end{array}
 & \quad
 \begin{array}{lll}
\mu(x_2,  x_2) &=& a\ x_2, \ \\
 \mu(x_2,   x_3)&=& b\ x_3, \ \\
 \mu(x_3,  x_2)&=&  \mu(x_3, x_3)=0,
  \end{array}
\end{array}
$$

$$  \alpha (x_1)= a\ x_1, \quad
 \alpha (x_2) =a\ x_2 , \quad
   \alpha (x_3)=b\ x_3,
$$
where $a,b$ are parameters in $\mathbb{K}$. The algebras are not associative
when $a\neq b$ and $b\neq 0$, since
$ \mu(\mu(x_1,  x_1), x_3)-  \mu(x_1,
\mu(x_1, x_3))=(a-b)b x_3.$

Let $\mathcal{R}$  be the  operator defined with respect to the basis $\{x_1,x_2,x_3\}$ by
$$
\mathcal{ R}(x_1)=  \rho_1 x_3,\
\mathcal{R}( x_2 )=\rho_2 x_3, \
\mathcal{R}(x_3 )  =0,
$$
where $\rho_1,\rho_2$ are  parameters in $\mathbb{K}$.  Then $\mathcal{R}$ is a Rota-Baxter operator  of weight $0$ on the Hom- associative algebra $(A,\mu ,\a)$.

\end{exam}

\begin{exam}
Let $(A,\mu,\a)$ be a Hom-associative algebra and $(M,l,r,\phi)$ be an $A$-bimodule.  We can easily, verify that $A\oplus M$ is  an $A$-bimodule under the following actions:
\begin{align*}
    a\bullet (b,m)=(\mu(a,b), a \cdot m), \quad (b,m) \bullet a=(0, m \cdot a), \forall a,b \in A,\ m \in M.
\end{align*}
Define the linear map $T: A \oplus M \to A,\ (a,m)\mapsto a$. Then $T$ is an $\O$-operator on $A$ with respect to the bimodule $A\oplus M$.

\end{exam}

\begin{prop}
A linear map $T : M \rightarrow A$ is an $\mathcal{O}$-operator on a Hom-associative algebra $A$ with respect to the $A$-bimodule $M$ if and only if the graph of $T$,
\begin{align*}
\mathrm{Gr}(T) = \{ (T(u), u ) |~ u \in M \}
\end{align*}
is a subalgebra of the semi-direct product algebra $A \oplus M$.
\end{prop}

In the following proposition, we show that an $\O$-operator
can be lifted up the  Rota-Baxter operator.

\begin{prop}
Let $(A,\mu,\a)$ be a Hom-associative algebra, $(M,l,r,\phi)$ be an $A$-bimodule and $T:M \to A$ be a linear map. Define $\widehat{T} \in End(A\oplus M)$  by $\widehat{T}(a,m)=(T(m),0)$.  Then $T$ is an $\O$-operator if and only $\widehat{T}$  is a Rota-Baxter operator on $A\oplus M$.
\end{prop}

Another characterization of an $\mathcal{O}$-operator can be given in terms of Nijenhuis operator on Hom-associative algebras.

\begin{defn}
Let $(A , \mu,\a)$ be a Hom-associative algebra. A linear map $N : A \rightarrow A$ is said to be a Nijenhuis operator if $N \circ \a=\a \circ N$ and  its Nijenhuis torsion vanishes, i.e.
\begin{align*}
N(x) \cdot N(y) = N ( N(x) \cdot y + x \cdot N(y) - N (x \cdot y) ), ~~ \text{ for all } x, y \in A.
\end{align*}
\end{defn}
Note that the deformed multiplication  $\mu_N : A \otimes A \rightarrow A$ given by $$ \mu_N(x, y) = N(x) \cdot y + x \cdot N(y) - N (x \cdot y),$$ defines a new Hom-associative multiplication on $A$, and $N$ becomes an algebra morphism from $(A, \mu_N,\a)$ to $(A, \mu,\a).$
In addition $\mu$ and $\mu_N$ are compatible, i.e. $\mu+\lambda \mu_N$ still a Hom-associative structure for each $\lambda \in \mathbb{K}$. Then the pair $(\mu,\mu_N)$ becomes a Hom-quantum bi-Hamiltonian system.

The following result is straightforward, so we omit details.
\begin{prop}\label{o-nij}
A linear map $T: M \rightarrow A$ is an $\mathcal{O}$-operator on $(A,\cdot,\a)$ with respect to the bimodule $(M,l,r,\phi)$ if and only if $N_T = \begin{pmatrix}
0 & T \\
0 & 0
\end{pmatrix} : A \oplus M \rightarrow A \oplus M$ is a Nijenhuis operator on the Hom-associative  algebra $A \oplus M$.
\end{prop}
Dendriform structures were first introduced by Loday in his study of the periodicity phenomenons in algebraic $K$-theory \cite{loday}. They become rather popular  in the last 20 years due to their connection with Rota-Baxter algebras, shuffle algebras and combinatorics \cite{aguiar, uchino2, guo-adv}.  In \cite{hom dend}, Makhlouf introduced the notion of Hom-dendriform algebra.

\begin{defn}
A Hom-dendriform algebra is a vector space $D$ together with three linear maps $\prec, \succ : D \otimes D \rightarrow D$ and $\a: D \to D$ satisfying the following three identities
\begin{align}
&( x \prec y ) \prec \a(z) =~ \a(x) \prec ( y \prec z + y \succ z ), \label{homdend-def-1}\\
&( x \succ y ) \prec \a(z) =~ \a(x) \succ ( y \prec z),\\
&( x \prec y + x \succ y ) \succ \a(z) =~ \a(x) \succ ( y \succ z), \label{homdend-def-3}
\end{align}
for all $x , y, z \in D$.
\end{defn}

It follows from (\ref{homdend-def-1})-(\ref{homdend-def-3}) that the new product $\star : D \otimes D \rightarrow D$ defined by $x \star y = x \prec y ~+~ y \succ x$ turns out to be Hom-associative. Thus, a Hom-dendriform algebra can be seen as a splitting of a Hom-associative algebra.

An $\mathcal{O}$-operator has an underlying Hom-dendriform structure on the bimodule.

\begin{prop}\label{o-dend}
Let $T : M \rightarrow A$ be an $\mathcal{O}$-operator on a Hom-associative algebra $A$ with respect to the bimodule $M$. Then the vector space $M$ carries a Hom-dendriform structure with
\begin{align*}
u \prec v =r(T(v))u ~~~ \text{ and } ~~~  u \succ v = l(T(u ))v, ~~~ \text{ for all } u , v \in M.
\end{align*}
Moreover,  $(M,\star_T=\prec+\succ, \phi)$ is a Hom-associative algebra.
\end{prop}

\subsection{Graded Lie algebras, Lift and Bidegree } 
In this section, we discuss a graded Lie algebra related to Hom-associative algebras and its corresponding Maurer-Cartan equation for which a solution characterize an $\mathcal{O}$-operator. It is shown in   \cite{amm-ej-makh} that  a Gerstenhaber algebra structure is constructed on the cohomology of  Hom-associative algebras using a  cup product together with the degree  $-1$ graded Lie bracket.
We aim to show that an $\mathcal{O}$-operator can be characterized by certain solutions of the corresponding Maurer-Cartan equation (see Theorem \ref{new-gla}).

\begin{defn}
  Let $(G=\oplus_{k=1}^\infty G_i,[\cdot,\cdot], d)$ be a differential graded Lie algebra.  A degree $1$ element $\theta\in G_2$ is called a Maurer-Cartan element of $G$ if it
  satisfies the following Maurer-Cartan equation:
  \begin{equation}
  d \theta+ \frac{1}{2} [\theta,\theta]=0.
  \label{eq:mce}
  \end{equation}
  \end{defn}
A graded Lie algebra is a differential graded Lie
algebra with $d=0$. Then we
have

\begin{prop}
Let $(G=\oplus_{k=1}^\infty G_i,[\cdot,\cdot])$ be a graded Lie algebra and let $\mu\in G_2$ be a Maurer-Cartan element. Then the map
$$ d_\mu: G \longrightarrow G, \ d_\mu(\nu):=[\mu, \nu], \quad \forall \nu\in G,$$
is a differential on $G$. For any $\nu\in  G$, the sum $\mu+\nu$ is a
Maurer-Cartan element of the graded Lie algebra $(G, [\cdot,\cdot])$ if and only if $\nu$ is a Maurer-Cartan element of the differential graded Lie algebra $(G, [\cdot,\cdot], d_\mu)$. \label{pp:mce}
\end{prop}

Let $A$ be a vector space and $\a: A \to A$  be a linear map.  Denote by  $V_n^{\a}$ the set of all linear maps $f:A^{\otimes n} \to A$ satisfying
$$\a (f(a_1,\cdots,a_{n}))=f(\a(a_1),\cdots,\a(a_{n})),~~~~\textrm{for all}\ a_i \in A.$$
The graded space $G=\oplus_{n\geq 1} V_n^\a$  carries a graded  Lie  algebra structure
$[-,-]_\a: V_m^{\a} \times V_n^{\a} \to V_{m+n-1}^{\a}$ given by
{\small$$[f,g]_\a=f \circ g-(-1)^{(m-1)(n-1)}g \circ f,~~~~\textrm{for all }\ f \in V_m^{\a}, g \in V_n^{\a},$$}
where
\begin{align*}
& (f \circ g ) (a_1, \ldots, a_{m+n-1} ) = \\ &  \sum_{i = 1}^{m} (-1)^{(i-1)(n-1)}~f ( \a^{n-1}(a_1), \ldots, \a^{n-1}(a_{i-1}), g ( a_i, \ldots, a_{i+n-1}), \ldots, \a^{n-1}(a_{m+n-1})).
\end{align*}
Let $\mu \in V_2^\a$. By a direct computation, we have
\begin{align*}
\frac{1}{2}[\mu,\mu]_\a(x,y,z)=\mu \circ \mu(x,y,z)=\mu(\mu(x,y),\a(z))-\mu(\a(x),\mu(y,z)),
\end{align*}
which implies that $\mu$ defines a Hom-associative structure on $A$ if an only if $[\mu,\mu]_\a=0$, i.e. $\mu$ is a Maurer-Cartan element of the graded Lie algebra
$G$.\\ 

In the following, we introduce the notion of bidegree and discuss it with respect to Gerstenhaber bracket.
Let $\A_{1}$ and $\A_{2}$ be two vector spaces 
and  $c:\A^{\ot n}_{2}\to\A_{1}$ be a linear map,
or a cochain in $C^{n}(\A_{2},\A_{1})$.
We  construct a cochain
$\hat{c}\in C^{n}(\A_{1}\oplus\A_{2})$ by
$$
\hat{c}\Big((a_{1},x_{1})\ot...\ot(a_{n},x_{n})\Big)
:=(c(x_{1},...,x_{n}),0).
$$
In general, for a given multilinear map
$f:\A_{i_1}\ot\A_{i_2}\ot...\ot\A_{i_n}\to\A_{j}$,
$i_1,...,i_n,j\in\{1,2\}$,
we define a cochain $\hat{f}\in C^{n}(\A_{1}\oplus\A_{2})$ by
\begin{eqnarray*}
\hat{f}:=
\left\{
\begin{array}{ll}
f & \text{on $\A_{i_1}\ot\A_{i_2}\ot...\ot\A_{i_n}$,} \\
0  & \text{all other cases.}
\end{array}
\right.
\end{eqnarray*}
We call the cochain $\hat{f}$ a horizontal
lift of $f$, or simply a lift.
For instance, the lifts
of $\a:\A_{1}\ot\A_{1}\to\A_{1}$,
$\b:\A_{1}\ot\A_{2}\to\A_{2}$ and
$\gamma:\A_{2}\ot\A_{1}\to\A_{2}$
are defined  respectively by
\begin{eqnarray}
\label{exabg1}\hat{\a}((a,x),(b,y))&=&(\a(a,b),0),\\
\label{exabg2}\hat{\b}((a,x),(b,y))&=&(0,\b(a,y)),\\
\label{exabg3}\hat{\gamma}((a,x),(b,y))&=&(0,\gamma(x,b)).
\end{eqnarray}
Let $H:\A_{2}\to\A_{1}$ (resp. $H:\A_{1}\to\A_{2}$)
be a 1 cochain. The lift is defined by
$$
\what{H}(a,x)=(H(x),0) \ \ \ (\text{resp. $\what{H}(a,x)=(0,H(a))$}).
$$
For any $(a,x)\in\A_{1}\oplus\A_{2}$,

We denote by $\A^{l,k}$
the direct sum of all $(l+k)$-tensor powers of $\A_{1}$ and $\A_{2}$,
where $l$ (resp. $k$) is the number of $\A_{1}$ (resp. $\A_{2}$).
For instance,
$$
\A^{1,2}:=
(\A_{1}\ot\A_{2}\ot\A_{2})\oplus
(\A_{2}\ot\A_{1}\ot\A_{2})\oplus
(\A_{2}\ot\A_{2}\ot\A_{1}).
$$
The tensor space $(\A_{1}\oplus\A_{2})^{\ot n}$
is expanded into the direct sum of $\A^{l,k}$, $l+k=n$.
For instance,
$$
(\A_{1}\oplus\A_{2})^{\ot 2}=
\A^{2,0}\oplus\A^{1,1}\oplus\A^{0,2}.
$$
Let $f$ be a $n$-cochain in $C^{n}(\A_{1}\oplus\A_{2})$.
We say that the {\em bidegree} of $f$ is $k|l$
if $f$ is an element in $C^{n}(\A^{l,k-1},\A_{1})$ or
in $C^{n}(\A^{l-1,k},\A_{2})$, where $n=l+k-1$.
We denote the bidegree of $f$ by $||f||=k|l$.
In general, cochains do not have bidegree.
We call a cochain $f$ a {\em homogeneous cochain},
if $f$ has a bidegree.\\
\indent
We have $k+l\ge 2$, because $n\ge 1$.
Thus there are no cochains of
bidegree $0|0$ or $1|0$ or $0|1$.
We recall
$\hat{\a},\hat{\b},\hat{\gamma}\in C^{2}(\A_{1}\oplus\A_{2})$
in (\ref{exabg1}), (\ref{exabg2}) and (\ref{exabg3}).
One can easily see
$||\hat{\a}||=||\hat{\b}||=||\hat{\gamma}||=1|2$.

\begin{lemma}
If $||f||=k|0$ (resp. $0|k$) and $||g||=l|0$ (resp. $0|l$),
then $[f,g]_\a=0$.
\end{lemma}

\begin{prop}\label{derivedbracket} Let $f,g\in C^{\bullet}(\A_{1}\oplus\A_{2})$.
If $||f||=k_{f}|l_{f}$ and $||g||=k_{g}|l_{g}$,
then the Gerstenhaber bracket $[f,g]_\a$ has the bidegree
$k_{f}+k_{g}-1|l_{f}+l_{g}-1$.
\end{prop}

\section{Cohomologies of $\O$-operators on Hom-associative algebras}

\subsection{Cohomology of $\mathcal{O}$-operator}

Let $A$ and $M$ be two vector spaces equipped with maps $\mu : A^{\otimes 2} \rightarrow A$, $l : A \otimes M \rightarrow M$,  $r : M \otimes A \rightarrow M$, $\a:A\to A$ and $\phi:M\to M$.  Set $V= A \oplus M$. It is, naturally equipped with the linear map $\a+\phi: (x,u)\mapsto (\a(x),\phi(u))$.  Now,
consider the graded Lie algebra structure on $\mathfrak{g} = \oplus_{n \geq 1} V_n^{\a+\phi}$ Observe that the elements $\mu , l, r \in V_2^{\a+\phi}$. Therefore, $\mu + l + r \in V_2^{\a+\phi}$

\begin{prop}\label{mlr}
The product $\mu$ defines a Hom-associative structure on $A$ and $l, r$ define an $A$-bimodule structure on $M$ with respect to $\phi$ if and only if $\mu + l + r \in V_2^{\a+\phi}$ is a Maurer-Cartan element in $\mathfrak{g}.$
\end{prop}

\begin{proof}
By  a direct computation, we have
\begin{align*}
&\frac{1}{2}[\mu + l + r,  \mu + l + r ]_\a((a_1 , m_1), (a_2 , m_2), (a_3, m_3)) \\
=&(\mu + l + r) \circ ( \mu + l + r )((a_1 , m_1), (a_2 , m_2), (a_3, m_3)\\
 =&(\mu + l + r) \big( (\mu+ l + r) ((a_1 , m_1), (a_2 , m_2)), (\a(a_3), \phi(m_3)) \big) \\
&- (\mu + l + r) \big(  (\a(a_1), \phi(m_1)), (\mu+ l + r ) ((a_2 , m_2), (a_3 , m_3)) \big)\\
=&\big( \mu(\mu(a_1 ,a_2) ,\a(a_3)),~  l(\mu(a_1, a_2),\phi( m_3)) + r(l(a_1, m_2), \a(a_3)) + r(r(m_1, a_2 ), \a(a_3)) \big)\\
& - \big(  \mu(\a(a_1),\mu (a_2 ,a_3)),~ l(\a(a_1),l (a_2, m_3)) + l(\a(a_1) ,r(m_2 ,a_3) ) + r(r(m_1, a_2 ) ,\a(a_3))   \big).
\end{align*}
Then $\mu$ defines a Hom-associative structure on $A$ and $l, r$ define an $A$-bimodule structure on $M$ with respect to $\phi$ if and only if $\mu + l + r $ is a Maurer-Cartan element in $\mathfrak{g}.$
\end{proof}

We aim in the following the Voronov's derived bracket construction \cite{voro} (see also \cite{uchino3}) to get a graded Lie algebra structure on $\oplus_{n \geq 1} \text{Hom} (M^{\otimes n}, A )$, where $A$ is  a regular Hom-associative algebra and $(l, r)$ defines an $A$-bimodule structure on $M$ with respect to $\phi$ supposed to be invertible.

According to  the above proposition, the graded Lie algebra $(\mathfrak{g}, [~, ~] )$ together with the differential $d_{\mu + l + r } = [ \mu + l + r, ~ ]_\a$ becomes a differential graded Lie algebra.
Moreover, it is easy to see that the graded subspace $\oplus_{n \geq 1} \text{Hom} (M^{\otimes n}, A )$ is an abelian subalgebra. 

Define, for $P \in \text{Hom}(M^{\otimes m}, A), ~ Q \in \text{Hom}(M^{\otimes n}, A)$, $m , n \geq 1$,  the  bracket
\begin{align}
\llbracket P, Q \rrbracket_\a := (-1)^m [ [ \mu + l + r , P ]_\a, Q ]_\a,
\end{align}
It is well defined according to Proposition \ref{derivedbracket}, $\llbracket -, - \rrbracket_\a$ and may be explicitly given by
\begin{align}\label{derived-bracket}
&\llbracket P, Q \rrbracket_\a  (u_1, \ldots, u_{m+n}) \\
&= \sum_{i=1}^m (-1)^{(i-1) n} P (\phi^m(u_1), \ldots, \phi^m(u_{i-1}), l(Q (u_i, \ldots, u_{i+n-1}),\phi^{m-1}( u_{i+n})), \ldots, \phi^{m}(u_{m+n} )) \nonumber\\
~&- \sum_{i=1}^m (-1)^{in} P (\phi^{m}(u_1), \ldots, \phi^{m}(u_{i-1}), r(\phi^{m-1}(u_i), Q (u_{i+1}, \ldots, u_{i+n})), \phi^{m}(u_{i+n+1}), \ldots, \phi^{m}(u_{m+n})) \nonumber\\
&- (-1)^{mn} \bigg\{ \sum_{i=1}^n (-1)^{(i-1)m}~ Q (\phi^{n}(u_1), \ldots, \phi^{n}(u_{i-1}), l(P (u_i, \ldots, u_{i+m-1}), \phi^{n-1}(u_{i+m})), \ldots, \phi^{n}(u_{m+n})) \nonumber\\
&- \sum_{i=1}^n (-1)^{im}~ Q (\phi^{n}(u_1), \ldots, \phi^{n}(u_{i-1}), r(\phi^{n-1}(u_i), P (u_{i+1}, \ldots, u_{i+m})) , \phi^{n}(u_{i+m+1}), \ldots, \phi^{n}(u_{m+n}) ) \bigg\} \nonumber\\
&+ (-1)^{mn} \bigg\{\mu( P(\phi^{n-1}(u_1), \ldots, \phi^{n-1}(u_m)) , Q (\phi^{m-1}(u_{m+1}), \ldots, \phi^{m-1}(u_{m+n}) ))
 \nonumber\\&- (-1)^{mn}~\mu( Q (\phi^{m-1}(u_1), \ldots, \phi^{m-1}(u_n)) , P ( \phi^{n-1}(u_{n+1}), \ldots, \phi^{n-1}(u_{m+n}) )) \bigg\}, \nonumber
\end{align}
for $u_1, \ldots, u_{m+n} \in M.$ We may extend the graded Lie bracket to $\oplus_{n \geq 0} \text{Hom} (M^{\otimes n}, A )$ as follows.
For $P \in \text{Hom}(M^{\otimes m}, A)$ and $a \in A$, the bracket is
\begin{align}\label{derived-bracket-0}
&\llbracket P, a \rrbracket_\a (u_1, \ldots, u_m )  \\
=& \sum_{i=1}^m P (\phi^m(u_1), \ldots, \phi^m(u_{i-1}), l(a,\phi^{m-1}(u_i)) - r(\phi^{m-1}(u_i),a) , \phi^m(u_{i+1}), \ldots, \phi^m(u_{m}) ) \nonumber\\
~~&+\mu( P (\phi^{-1}(u_1), \ldots, \phi^{-1}(u_m)) ,a) - \mu(a, \cdot P (\phi^{-1}(u_1), \ldots, \phi^{-1}(u_m)) \nonumber
\end{align}
and for $a, b \in A$, we define $\llbracket a, b \rrbracket_\a$ to be the Lie bracket $[a, b ]_C = \mu(a,b) - \mu(b, a).$

Thus, for any $T, T' \in \text{Hom} (M, A)$, we have from (\ref{derived-bracket}) that
\begin{align}\label{t-t}
\llbracket T, T' \rrbracket _\a(u, v) = T (T'(u)v ) + T(u T'(v)) + T' (T(u)v) + T' (u T(v)) - T(u) \cdot T'(v) - T'(u) \cdot T(v).
\end{align}

For any $n \geq 0$, we define $C^n_\a (M, A) := \mathrm{Hom} (M^{\otimes n}, A)$ and consider the graded vector space
$C^\bullet (M, A ) = \oplus_{n \geq 0}  C^n_\a (M, A) =   \oplus_{n \geq 0}  \mathrm{Hom} (M^{\otimes n}, A)$. Thus, we obtain the following result.

\begin{thm}\label{new-gla}
The graded vector space $C^\bullet_\a (M, A ) = \oplus_{n \geq 0} \mathrm{Hom} (M^{\otimes n}, A)$ together with the above defined bracket $\llbracket ~, ~\rrbracket_\a$ form a graded Lie algebra. A linear map $T : M \rightarrow A$ is an $\mathcal{O}$-operator on $A$ with respect to the bimodule $M$ if and only if $T \in C^1_\a (M, A)$ is a Maurer-Cartan element in $(C^\bullet_\a (M, A ), \llbracket~,~\rrbracket_\a )$, i.e. $T$ satisfies $\llbracket T, T \rrbracket _\a= 0$.
\end{thm}

Therefore, from the general principle of Maurer-Cartan elements, we obtain.

\begin{thm}\label{new-gla-2}
Let $T : M \rightarrow A$ be an $\mathcal{O}$-operator on $A$ with respect to the bimodule $M$. Then $T$ induces a differential $d_T = \llbracket T, ~ \rrbracket_\a$ which makes the graded Lie algebra $( C^\bullet _\a(M, A ), \llbracket ~,~\rrbracket_\a)$ into a differential graded Lie algebra (dgLa).

Moreover, for any linear map $T' : M \rightarrow A$, the sum $T + T'$ is still an $\mathcal{O}$-operator if and only if $T'$ is a Maurer-Cartan element in the dgLa $( C^\bullet_\a (M, A ), \llbracket ~,~ \rrbracket_\a, d_T)$, that is 
\begin{align*}
\llbracket T + T' , T + T' \rrbracket _\a= 0  ~ \Longleftrightarrow ~ d_T (T') + \frac{1}{2} \llbracket T', T' \rrbracket _\a= 0.
\end{align*}
\end{thm}

It follows from the above theorem that if $T$ is an $\mathcal{O}$-operator, then $(C^\bullet _\a (M, A), d_T = \llbracket T, ~ \rrbracket_\a)$ is a cochain complex. Its cohomology is called \textbf{the cohomology of the $\mathcal{O}$-operator $T$}.

In the next paragraph, we interpret this cohomology as the Hochschild cohomology of a certain algebra with coefficients in a suitable bimodule. Then we will use the usual notation for Hochschild cohomology to denote the cohomology of an $\mathcal{O}$-operator.

\subsection{Cohomology of $\mathcal{O}$-operators as Hochschild cohomology}\label{sec3}

The aim of this section is to   show  that the cohomology of an $\mathcal{O}$-operator can also be described as the Hochschild cohomology of a certain Hom-associative algebra with coefficients in a  suitable bimodule. 

Let $T : M \rightarrow A$ be an $\mathcal{O}$-operator on $A$  with respect to the bimodule $(M,l,r,\phi)$. Then by Proposition \ref{o-dend} the vector space $M$ carries a Hom-associative algebra structure with the product
\begin{align}\label{m-ass}
u \star_T v = r(u, T(v )) + l(T(u) ,v), ~~~ \text{ for } u , v \in M.
\end{align}

According to \eqref{o-nij},  there is a Hom-associative algebra structure on
$A\oplus M \backsimeq M\oplus A$ defined   as
\begin{align*}
( x, u )  \cdot_{\overline{T}} (y, v)
=&(Tu ,0 )  \cdot_{(l,r)} (y, v)
 + (x, u)     \cdot_{(l,r)} (Tv, 0)
 - (T(l(x,v)+r(u,y) ,0)  \\
=&(Tu \cdot y , l(Tu,v) )
+ ( x \cdot Tv ,-r(u,Tv)  )
 - (T(l(x,v)+T(r(u,y)) , 0 )  \\
=&( r_T(u,y)+l_T(v,x) , u \star v ),
\end{align*}
where
\begin{equation}\label{rep on module V}
l_T(u,x) =Tu \cdot x-T(r(u,x)),\  r_T(x,u) =x \cdot Tu-T(l(x,u),\ \forall x \in A, u \in M.
\end{equation}
Therefore we have  the following result.

\begin{lemma}\label{new-rep-o}
Let $T: M \rightarrow A$ be an $\mathcal{O}$-operator on $A$ with respect to the bimodule $(M,l,r,\phi)$. Then   $(l_T, ~ r_T)$ defines an $M$-bimodule structure on $A$ with respect to $\a$.
\end{lemma}

By Lemma \ref{new-rep-o} we obtain an $M$-bimodule structure on the vector space $A$. Therefore, we may consider the corresponding Hochschild cohomology of $M$ with coefficients in $A$. We define
\begin{align*}
C^n_\a(M, A) := \text{Hom} ( M^{\otimes n}, A), ~~ \text{ for } n \geq 0
\end{align*}
and the differential  by
\begin{align}\label{zero-diff}
d_H (a) (m) = l_T (m, a) - r_T (a, m) = T(m) \cdot a - T (ma) - a \cdot T(m) + T (am),  ~~~ \text{ for } a \in A = C^0_\a (M, A)
\end{align}
and
\begin{align*}
&(d_H  f) (u_1, \ldots, u_{n+1}) =~ T ( \phi^{n-1}(u_1)) \cdot f (u_2, \ldots, u_{n+1})  - T (r(\phi^{n-1}(u_1), f(u_2, \ldots, u_{n+1}))) \\
~&+ \sum_{i=1}^n (-1)^i ~ f (\phi(u_1), \ldots, \phi(u_{i-1}), r(u_i ,T( u_{i+1})) + l(T(u_i) , u_{i+1}), \ldots, \phi(u_{n+1})) \\
~&+ (-1)^{n+1}~  f (u_1, \ldots, u_n ) \cdot T(\phi^{n-1}(u_{n+1}))  - (-1)^{n+1} T ( l(f(u_1, \ldots, u_n), \phi^{n-1}(u_{n+1})) ).
\end{align*}
We denote the group of $n$-cocycles by $Z^n_\a (M, A)$ and the group of $n$-coboundaries by $B^n_\a (M, A)$.
The corresponding cohomology groups are defined by $$H^n_\a(M, A) = Z^n_\a (M, A) / B^n_\a (M, A),\quad n \geq 0.$$
In particular, we have 
\begin{align*}
H^0 (M, A) =~& \{ a \in A |~ d_H (a) = 0 \} \\
=~& \{ a \in A |~ a \cdot T(m) - T(m) \cdot a = T(l(a,m) -r(m,a) ), ~ \forall m \in M \}.
\end{align*}
From this definition, it is easy to see that if $a, b \in H^0 (M, A)$, then their commutator $[a, b] := a \cdot b - b \cdot a$ is also in $H^0 (M, A)$. This shows that $H^0 (M, A)$ has a Hom-Lie algebra structure induced from that of $A$.

Note that a linear map $f : M \rightarrow A$ (i.e. $f \in C^1 (M, A)$) is closed if it satisfies
\begin{align}\label{1-coc}
T(u) \cdot f (v) + f (u) \cdot T(v) - T (u f(v) + f(u) v ) - f (uT(v) + T(u)v ) = 0,
\end{align}
for all $u, v \in M.$

For an $\mathcal{O}$-operator $T$ on $A$ with respect to the bimodule $M$, we get two coboundary operators $d_T = \llbracket T, ~\rrbracket_\a$ and $d_H$ on the same graded vector space $C^\bullet _\a(M, A) = \oplus_{n \geq 0} C^n_\a (M, A)$. The following proposition relates the above two coboundary operators.

\begin{prop}
Let $T : M \rightarrow A$ be an $\mathcal{O}$-operator on $A$ with respect to the $A$-bimodule $M$. Then the two coboundary operators are related by
\begin{align*}
d_T f = (-1)^n~ d_H f , ~~~~ \text{ for } f \in C^n_\a (M, A).
\end{align*}
\end{prop}

\begin{proof}
For any $f \in C^n_\a (M, A) = \text{Hom}(M^{\otimes n}, A )$ and $u_1, \ldots, u_{n+1} \in M$, we have from (\ref{derived-bracket}) that
\begin{align*}
(d_T f) (u_1, \ldots, u_{n+1}) =~& \llbracket T, f \rrbracket _\a(u_1, \ldots, u_{n+1})   \\
=~& T (l(f(u_1, \ldots, u_n) \phi^{n-1}(u_{n+1}))) - (-1)^n ~ T (r(\phi^{n-1}(u_1), f(u_2, \ldots, u_{n+1})) \\
-~& (-1)^n \bigg\{ \sum_{i=1}^n (-1)^{i-1}~ f (\phi(u_1), \ldots, \phi(u_{i-1}), T(u_i) u_{i+1} , u_{i+2}, \ldots, \phi(u_{n+1})) \\
-~& \sum_{i=1}^n (-1)^i f (\phi(u_1), \ldots, \phi(u_{i-1}), u_i T(u_{i+1}), u_{i+2}, \ldots, \phi(u_{n+1}) ) \bigg\} \\
+~& (-1)^n ~T(\phi^{n-1}(u_1)) \cdot f (u_2, \ldots, u_{n+1}) - f (u_1, \ldots, u_n) \cdot T(\phi^{n-1}(u_{n+1}))\\
=~& (-1)^n ~ (d_H f) (u_1, \ldots, u_{n+1}).
\end{align*}
The same holds true when $f = a \in A$. Compare (\ref{derived-bracket-0}) with (\ref{zero-diff}). Hence the proof.
\end{proof}

This shows that the cohomology of the either complex $(C^\bullet _\a(M, A), d_T)$ or $(C^\bullet (M, A), d_H)$ are isomorphic. Thus, we may use the same notation $H^\bullet _\a (M, A)$ to denote the cohomology of an $\mathcal{O}$-operator.

\section{Deformations of $\mathcal{O}$-operators}\label{sec4}

In this section, we study the Infinitesimal and formal deformation theory  of $\mathcal{O}$-operators following the classical approaches initiated by Gerstenhaber \cite{gers,gers2}.
\subsection{Infinitesimal deformations}
Let $(A, \cdot_A, \a_A )$ be a Hom-associative algebra and $(V,l_A,r_A,\phi)$ an $A$-bimodule, and $(B, \cdot_B,\a_B)$ be a Hom-associative algebra with $(W,l_B,l_B,\psi)$ a $B$-bimodule. Suppose $T : V \rightarrow A$ is an $\mathcal{O}$-operator on $A$ with respect to the $A$-bimodule $(V,l_A,r_A,\phi)$ and $T' : W\rightarrow B$ is an $\mathcal{O}$-operator on $B$ with respect to the bimodule $(B, \cdot_B,\a_B)$.

\begin{defn}\label{o-op-morphism}
A  morphism of $\mathcal{O}$-operators from $T$ to $T'$ consists of a pair $(f, g)$ of an algebra morphism $f : A \rightarrow B$ and a linear map $g : V \rightarrow W$ satisfying
\begin{align}
\psi g =~& g\phi, \label{eq0}\\
T'  g =~& f  T, \label{eq1}\\
l_B(f (x) ,g (u )) =~& g ( l_A(x,u) ), \label{eq2}\\
r_B(g (u ),f (x) ) =~& g ( r_A(u,x ),  \label{eq3}
\end{align}
for all  $x \in A$ and $ u \in V$.

\end{defn}
It is called an isomorphism if both $f$ and $g$ are  bijective. According to Lemma \ref{o-dend}, the vector spaces $V$ and $W$ can be endowed with a Hom-dendriform  algebra structures and $g$ becomes then a morphism between Hom-dendriform algebras $V$ and $W$ and also a morphism between the sub-adjacent Hom-associative algebras .

The proof of the following result is straightforward so we omit the details.
\begin{prop}
A pair of linear maps $(f : A \rightarrow B, ~ g : V \rightarrow W)$ is a morphism of $\mathcal{O}$-operators from $T$ to $T'$ if and only if
\begin{align*}
\mathrm{Gr} ((f, g)) := \big\{ \big( (a, u), (f(a), g (u)) \big) | ~ a \in A, u \in V \big\} \subset (A \oplus V) \oplus (B \oplus W)
\end{align*}
is a Hom-subalgebra, where $A \oplus V$ and $B \oplus W$ are equipped with semi-direct product Hom-algebra structures.
\end{prop}
In the rest of the paper, we will be most interested in morphisms between $\mathcal{O}$-operators on the same Hom-algebra with respect to the same bimodule.

Let $T: M \rightarrow A$ be an $\mathcal{O}$-operator on a Hom-associative algebra $A$ with respect to the $A$-bimodule $M$.
\begin{defn}
An infinitesimal  (one-parameter)  deformation of $T$ consists of a sum $T_t = T + t  \mathcal{T}$, for some $\mathcal{T} \in \text{Hom} (M,A)$, such that $T_t$ is an $\mathcal{O}$-operator on $A$ with respect to the bimodule $M$, for all $t$. In such a case, we say that $\mathcal{T}$ generates a infinitesimal or a linear deformation of $T$.
\end{defn}
It is easy to check that  $T_t=T+t \mathcal{T}$ is an infinitesimal deformation of an $\mathcal{O}$-operator $T$ if and only if for any $u,v \in M$,
\begin{align*}
T_t \phi =&\a T_t, ~~~~ \text{ for all } t , \\
T_t (u) \cdot T_t (v) =& T_t ( l(T_t (u)) v+ r(T_t(v))u ), ~~~~ \text{ for all } t \text{ and } u, v \in M.
\end{align*}
By equating coefficients of $t$ and $t^2$ from both side, we obtain
\begin{align}
\mathcal{T} \phi =~&\a \mathcal{T},  \label{lin def 1} ,\\
T(u) \cdot\mathcal{T} (v) +\mathcal{T}(u) \cdot T(v) =~& T (r(\mathcal{T}(v))u +l(\mathcal{T}(u))v) +\mathcal{T} (r(T(v))u + l(T(u)) v), \label{lin def 2}\\
\mathcal{T}(u) \cdot\mathcal{T}(v) =~&\mathcal{T} ( r(\mathcal{T}(v))u +l(\mathcal{T}(u)) v ). \label{lin def 3}
\end{align}
 The identity \eqref{lin def 2} means that  $d_T (\mathcal{T}) = 0$, that is $\mathcal{T}$ is a $1$-cocycle with respect to  the cohomology of the $\mathcal{O}$-operator $T$. Eqs. \eqref{lin def 1} and \eqref{lin def 3} mean that $\mathcal{T}$ is an $\mathcal{O}$-operator on $A$ associated  to the bimodule $M$.

Given a  Hom-dendriform algebra $(D,\prec,\succ,\phi)$.  Let $\omega_\prec, \omega_\succ: D\otimes D \to D$ be two linear maps.  If for any $t \in \mathbb{K}$, the multiplications $\prec_t$ and $\succ_t$ defined by
\begin{align*}
u \prec_t v=u \prec v+t  \omega_\prec(u,v),\quad
u \succ_t v=u \succ v+t  \omega_\succ(u,v),\ \forall u,v \in D
\end{align*}
give a Hom-dendriform algebra structure on $D$, we say that the pair $(\omega_\prec, \omega_\succ)$ generates a (one parameter) infinitesimal deformation of the Hom-dendriform algebra $(D,\prec,\succ,\phi)$.

In the following, we show that an infinitesimal deformation of an $\mathcal{O}$-operator induces an infinitesimal deformation of the corresponding Hom-dendriform structure on the module.

\begin{prop}
If $\mathcal{T}$ generates an infinitesimal deformation of an $\mathcal{O}$-operator $T$ on a Hom-associative algebra $A$ with respect to the bimodule $M$, then the multiplications
\begin{align*}
u \prec_t  v = u T(v) + t~ u\mathcal{T}(v), ~ u \succ_t v = T(u) v + t~\mathcal{T}(u) v
\end{align*}
define an infinitesimal deformation of the corresponding Hom-dendriform structure on $M$.
\end{prop}

\begin{corollary}
If $\mathcal{T}$ generates an infinitesimal deformation of an $\mathcal{O}$-operator $T$ on a Hom-associative algebra $A$ with respect to the bimodule $M$, then the multiplication
\begin{align*}
u \star_t  v =  u \star v +t  ( u\mathcal{T}(v) +\mathcal{T}(u) v)
\end{align*}
defines an infinitesimal deformation of the corresponding Hom-associative structure on $M$.
\end{corollary}

\begin{defn}
Let $T$ be an $\mathcal{O}$-operator on a Hom-associative algebra $A$ with respect to a bimodule $(M,l,r,\phi)$.
Two infinitesimal deformations $T^1_t = T + t\mathcal{T}_1$ and $T^2_t = T + t\mathcal{T}_2$ of $T$ are  equivalent if there exists an $x \in A$ such that  $\a(x)=x$ and the pair $\big( \text{id}_A + t (L_x - R_x), \text{id}_M + t (l_x - r_x) \big)$ is a morphism of $\mathcal{O}$-operators from $T^1_t$ to $T^2_t$.
\end{defn}
Let us recall  from Definition \ref{o-op-morphism} that the pair  $\big( \text{id}_A + t (L_x - R_x), \text{id}_M + t (l_x - r_x) \big)$ is a morphism of $\mathcal{O}$-operators from $T^1_t$ to $T^2_t$ if  the following conditions are satisfied
\begin{itemize}
  \item [(i)]  The map $\text{id}_A + t (L_x - R_x)$ is a Hom-associative algebra homomorphism.
  \item [(ii)]  $\phi \circ ( \text{id}_M + t (l_x - r_x) )= (\text{id}_M + t (l_x - r_x))\circ  \phi$.
  \item [(iii)] $(T + t\mathcal{T}_2)\circ (\text{id}_M + t (l_x - r_x))=(\text{id}_A + t (L_x - R_x)) \circ  (T + t\mathcal{T}_1)$.
  \item [(iv)]  $l(y+t[x,y],u+t(l(x,u)-r(u,x)))=l(y,u)+t(l(x,l(y,u))-r(l(y,u),x))$ .
  \item [(v)] $r(u+t(l(x,u)-r(u,x)),y+t[x,y])=r(u,y)+tl(x,r(u,y))-tr(r(u,y),x)$.
\end{itemize}
From (i), we obtain, for all $y,z \in A$
\begin{align*}
(\text{id}_A + t (L_x - R_x)) (y \cdot z) = (\text{id}_A + t (L_x - R_x))(y) \cdot (\text{id}_A + t (L_x - R_x))(z),
\end{align*}
or equivalently
\begin{align}
(x\cdot y-y\cdot x) \cdot (x\cdot z-z \cdot x)= 0, ~~\forall y,z \in A. \label{comm-comm-zero}
\end{align}
Since $\a(x)=x$, then condition (ii) holds.  Moreover,  from (iii) we get, for any $u \in M$,
\begin{align*}
 (T + t\mathcal{T}_2 ) \circ (\text{id}_M + t (l_x - r_x)) (u) = (\text{id}_A + t (L_x - R_x)) \circ (T + t\mathcal{T}_1 ) (u),
 \end{align*}
On comparing the coefficients of $t$ and $t^2$, respectively, from both sides of the above identity, we obtain
\begin{align}
\mathcal{T}_1 (u) -\mathcal{T}_2 (u) =~& T ( l(x,u) - r(u,x) ) - (x \cdot T(u) - T(u) \cdot x) \nonumber\\
=~& l_T (u, x) - r_T (x, u), \label{1-cocycle-linear}\\
x \cdot\mathcal{T}_1 (u) -\mathcal{T}_1 (u) \cdot x =~&\mathcal{T}_2 ( l(x,u) -r(u,x) ).
\label{T1-T2}
\end{align}
From (iv),  on comparing just the coefficients of $t^2$ (since the coefficients of $t$ vanish) from both sides of the identity, we obtain  $$l(x\cdot y-y \cdot x,l(x,u) -r(u,x)) = 0 $$ or equivalently,
\begin{align}
l_{x\cdot y-y \cdot x} \circ l_x = l_{x\cdot y-y \cdot x} \circ r_x, ~~~~ \text{ for } y \in A.
\label{ll-lr}
\end{align}
Similarly,  (v)  gives rise to
\begin{align}
r_{x\cdot y-y \cdot x} \circ l_x = r_{x \cdot y-y \cdot x} \circ r_x, ~~~~ \text{ for } y \in A.
\label{rl-rr}
\end{align}

\begin{thm}
Let  $T^1_t = T + t\mathcal{T}_1$ and $T^2_t = T + t\mathcal{T}_2$ be two equivalent infinitesimal  deformations of an $\mathcal{O}$-operator $T$. Then $\mathcal{T}_1$ and $\mathcal{T}_2$ defines the same cohomology class in $H^1_\a(M, A).$
\end{thm}

\begin{proof}
The proof follows from equation  \eqref{1-cocycle-linear}.
\end{proof}

\begin{defn}
Let $T: M \to A$ be an $\O$-operator on a Hom-associative algebra $A$ with respect to a representation $(M,l,r,\phi)$.
An infinitesimal  deformation $T_t = T + t\mathcal{T}$  is said to be trivial if it is equivalent to the deformation $T_0= T.$
\end{defn}
According to the above computation,  $T_t = T + t\mathcal{T}$ is trivial if there exists an element $x \in A$ such that $\a(x)=x$ and satisfying
\begin{align}
& (x\cdot y-y\cdot x) \cdot (x\cdot z-z \cdot x)= 0, \label{Nij element 1}\\
& x \cdot ( l_T (u, x) - r_T (x, u))-( l_T (u, x) - r_T (x, u))\cdot x=0, \label{Nij element 2} \\
& l_{x\cdot y-y \cdot x} \circ l_x = l_{x\cdot y-y \cdot x} \circ r_x,\label{Nij element 3} \\
&r_{x\cdot y-y \cdot x} \circ l_x = r_{x \cdot y-y \cdot x} \circ r_x, \label{Nij element 4}
\end{align}
for any $y,z  \in A$ and $u \in M$.  Then we introduce the notion of a Nijenhuis element associated to an $\O$-operator on a Hom-associative algebra.

\begin{defn}
Let $T$ be an $\O$-operator on a Hom-associative algebra $(A,\a)$ with respect to a representation $(M,l,r,\phi)$. An element $x \in A$ is called a Nijenhuis element associated to $T$ if $x$ satisfies $\a(x)=x$ and the conditions  \eqref{Nij element 1}--\eqref{Nij element 4}  hold.
\end{defn}
Let us denote the set of Nijenhuis elements associated to the $\O$-operator $T$ by$ \textrm{Nij}(T)$.

In \cite{O-op on hom-lie}, the authors introduce Nijenhuis elements associated to an $\O$-
operator on a Hom-Lie algebra to study their trivial deformations. If $T$ is an $\O$-operator on a Hom-Lie algebra $(A,[\cdot,\cdot],\a)$ with respect to a representation $(M,\rho,\b)$. Then an element $x \in A$ is called a Nijenhuis element if $\a(x)=x$ and
\begin{align*}
[[x,y],[x,z]]=0,\ \rho([x,y])\rho(x)=0,\  [x, T\rho(x)v+[Tv,x]]=0,
\end{align*}
for all $y,z \in A$ and $v \in M$.

\begin{prop}
Let $x \in  A$ be a Nijenhuis
element with respect to an  $\O$-operator $T$ on a Hom-associative $A$. Then $x$  is also a Nijenhuis element for the $O$-operator $T$ on the commutator Hom-Lie algebra $(A,[\cdot,\cdot]_C,\a)$.
\end{prop}

\begin{proof}
Straightforward.
\end{proof}

By Eqs. \eqref{comm-comm-zero}--\eqref{rl-rr}, it is obvious that a trivial infinitesimal deformation gives rise to a Nijenhuis element. Conversely, a Nijenhuis element can also generate a trivial infinitesimal  deformation as the following theorem shows.
%

\begin{thm}
Let $T$ be an $\O$-operator on a Hom-associative algebra $(A,\a)$ with respect to a representation $(M,l,r,\phi)$.  For any element $x \in \textrm{Nij}(T)$,  $T_t:= T+t \mathcal{T}$, with $\mathcal{T}=d_H(x)$, is a trivial   infinitesimal deformation  of $T$.
\end{thm}

\begin{proof}
First, we need to prove that $T_t:= T+t \mathcal{T}$ is a infinitesimal deformation generated by   $\mathcal{T}=d_H(x)$, where $x \in \textrm{Nij}(T)$.  That is $\mathcal{T}$ must satisfy conditions \eqref{lin def 1}, \eqref{lin def 2} and \eqref{lin def 3}. First,
\begin{align*}
\a \mathcal{T}(u)=\a d_H(x)(u)=d_H(x)(\phi(u))=\mathcal{T}(\phi(u)), \forall u \in M.
\end{align*}
Then $\mathcal{T}$ satisfies \eqref{lin def 1}.  The identity \eqref{lin def 2} holds immediately since $\mathcal{T}=d_H(x)$.
The identity \eqref{lin def 3} is also straightforward. We omit details since it  follows from a long computation. We refer to  \cite{tang} for a similar computation.
Finally, since $x$ is a Nijenhuis element, it follows that the pair $(id_A + t (L_x - R_x, id_M + t (l_x - r_x))$ is a morphism of $\mathcal{O}$-operators from $T_t$ to $T$. Hence the result follows.
\end{proof}

\subsection{Formal Deformations}
In this section, we study one-parameter formal  deformations of
O-operators.   Let $(A,\mu,\a)$ be a Hom-associative algebra with a representation $(M,l,r,\phi)$.
Let $\mathbb{K}[[t]]$  be the formal power series ring in one variable $t$ and $A[[t]]$ be the formal power series in $t$ with coefficients in $A$.  Then the tuple $(A[[t]],\mu_t,\a_t)$  is a Hom-associative algebra, where the product $\mu_t$ and the structure map $\a_t$ are obtained by extending $\mu$ and $\a$ linearly over the ring $\mathbb{K}[[t]]$.  Moreover the maps $l$, $r$ and $\phi$  can be extended linearly over $\mathbb{K}[[t]]$  to obtain $\mathbb{K}[[t]]$-linear maps $l_t: A[[t]] \otimes M[[t]] \to M[[t]],\ r_t: M[[t]] \otimes A[[t]] \to M[[t]]$ and $\phi_t: M[[t]] \to M[[t]]$. Then  $(M[[t]],l_t,r_t,\phi_t)$  is a representation of the Hom-associative algebra   $(A[[t]],\mu_t,\a_t)$.

\begin{defn}
Let $T$ be an $\O$-operator on a Home associative algebra $(A,\mu,\a)$ with respect to a representation $(M,l,r,\phi)$. A formal one-parameter deformation of $T$ is given by
\begin{align*}
T_t =T_0+ \sum_{i\geq 1} t^i T_i,\quad \textrm{with}\ T_0=T,\ T_i \in C^1_{\phi,\a}(M,A)
\end{align*}
such that $T_t: M[[t]] \to A[[t]] $ is an $\O$-operator on  $(A[[t]],\mu_t,\a_t)$ with respect to $(M[[t]],l_t,r_t,\phi_t)$.
\end{defn}
That is
\begin{align}
    & T_t \phi_t =\a_t T_t ,  \label{formal def 1} \\
    & T_t(u)\cdot T_t(v)=T_t(T_t(u)v+uT_t(v)),\ \forall u,v \in M. \label{formal def 2}
\end{align}
Condition \eqref{formal def 1} is immediate since $T_i \in C^1_{\phi,\a}(M,A),  \ \forall i$.
For any $k \geq 0$, by equating the coefficients of $t^k$ from both sides of equation \eqref{formal def 2}, we obtain the following system of equations
\begin{align}
    \sum_{i+j=k} T_i(u)\cdot T_j(v)= \sum_{i+j=k} T_i(T_j(u)v+uT_j(v)).
\end{align}
The above identity holds for $k=0$, since $T$ is an $\O$-operator.  For $k=1$, we get
\begin{align}
  T(u) \cdot T_1 (v) + T_1 (u) \cdot T(v) = T ( u T_1 (v) + T_1 (u) v ) + T_1 (u T(v) + T(u)v), \forall u, v \in M,
\end{align}
which is equivalent to
\begin{align}
    \llbracket T,T_1 \rrbracket_\a =0. \label{T1 is 1-cocycle}
\end{align}
That is $d_T(T_1)=0$. Therefore, we get the following proposition.
\begin{prop}\label{lin-term-1-co}
Let $T_t = \sum_{i \geq 0} t^i T_i$ be a formal deformation of an $\mathcal{O}$-operator $T$ on $A$ with respect to the bimodule $M$. Then the $1$-cochain  $T_1 \in C^1_{\phi,\a}(M,A)$ is a $1$-cocycle with respect to  the cohomology of the $\mathcal{O}$-operator $T$.
\end{prop}

The $1$-cochain $T_1$  is called the infinitesimal of the deformation $T_t$. More generally,
if $T_i = 0$ for $1 \leq i \leq (n -1)$  and $T_n$ is a non-zero cochain, then $T_n$ is called the $n$-infinitesimal of the deformation $T_t$.

In \cite{cohom hom dend}, the author study formal deformations of Hom-dendriform algebras and their relation with the cohomology.  Here we provide a connection between deformations of an $\mathcal{O}$-operator with deformations of the corresponding Hom-dendriform algebra. We mean by a deformation of a Hom-dendriform algebra, a formal deformation of its defining two multiplications.

\begin{prop}
Let $T_t = \sum_{i \geq 0 } t^i T_i$ be a formal deformation of an $\mathcal{O}$-operator $T$. Then the formal sums
\begin{align*}
u \prec_t v = \sum_{i \geq 0} t^i u T_i (v) \quad \text{ and } \quad
u \succ_t v = \sum_{i \geq 0} t^i T_i (u) v
\end{align*}
define a formal deformation of the Hom-dendriform structure on $M$.
\end{prop}

\begin{corollary}
It follows that the formal sum $u \star_t v = \sum_{i \geq 0} t^i ( u T_i (v) + T_i (u) v)$ defines a formal deformation of the Hom-associative structure on $M$.
\end{corollary}

Now, we consider the equivalence of two formal deformations of an $\O$-operator. The definition is motivated by  the  associative algebra case \cite{def O-op on ass al}.

\begin{defn}
Two formal deformations $T_t = \sum_{i \geq 0} t^i T_i$ and $\overline{T}_t = \sum_{i \geq 0} t^i \overline{T}_i$ of an $\mathcal{O}$-operator $T$ are said to be equivalent if there exists an element $x \in A$ satisfying $\a(x)=x$, linear maps $f_i: A \to A$ and $g_i: M \to M$, for $i \geq 2$, such that the pair $(f_t,g_t)$ given by
\begin{align*}
f_t = id_A + t (L_x - R_x) + \sum_{i \geq 2} t^i f_i \quad \text{ and } \quad
g_t = id_M + t (l_x - r_x) +  \sum_{i \geq 2} t^i g_i
\end{align*}
defines a formal morphism of $\mathcal{O}$-operators from $T_t$ to $\overline{T}_t$.
\end{defn}
With the above notations, for all $y,z \in A$ and $u \in M$, the equivalence of two deformations $T_t$ and  $\overline{T}_t $ corresponds to  the following conditions
\begin{align}
f_t  T_t=& T_t g_t, \label{(1)} \\
f_t(y \cdot z)=&f_t(y) \cdot f_t(z),\\
g_t(yu)=&f_t (y)g_t(u),\\
g_t(uy)=&g_t(u)f_t(y),\\
f_t  \a= \a f_t,\  &   g_t \phi =\phi g_t.
\end{align}
In particular, a formal deformation $T_t$ of an $\O$-operator $T$ is said to be trivial if it is equivalent to $T$ (here $T$ is regarded as a deformation of itself).

Now, on comparing the coefficients of $t$ from both sides of the condition \eqref{(1)}, we get
\begin{align*}
T_1 (u) - \overline{T}_1 (u) = T(u) \cdot x - T(ux) - x \cdot T(u) + T (xu) = d_H (x) (u).
\end{align*}
Consequently, we obtain the following result.
\begin{prop}
The infinitesimals of equivalent deformations are cohomologous, i.e. they lie on the same cohomology class .
\end{prop}

\begin{defn}
An $\mathcal{O}$-operator $T$ is said to be rigid if every deformation $T_t$ of $T$ is equivalent to the deformation $\overline{T}_t = T$.
\end{defn}

As a cohomological condition of the rigidity, we have the following result which
suggests that the rigidity of an $\O$-operator is a very strong condition.

\begin{prop}
Let $T$ be an $\mathcal{O}$-operator on a Hom-associative algebra $A$ with respect to a bimodule $M$. If $Z^1 (M, A) = d_H (\mathrm{Nij}(T))$ then $T$ is rigid.
\end{prop}

\begin{proof}
Let $T_t = \sum_{i \geq 0} t^i T_i$ be any formal deformation of the $\mathcal{O}$-operator $T$. Then by Proposition \ref{lin-term-1-co}, the linear term $T_1$ is in $Z^1(M, A)$. Therefore, by assumption, we have $T_1 = d_H (x)$, for some $x \in \text{Nij} (T)$. Set
\begin{align*}
f_t := id_A + t (L_x-R_x )  \quad \text{ and }  \quad g_t := id_M + t (l_x - r_x)
\end{align*}
and define $\overline{T}_t = f_t \circ T_t \circ g_t^{-1}$. Then $T_t$ is equivalent to $\overline{T}_t$. Moreover, from the definition of $\overline{T}_t$, we have
\begin{align*}
\overline{T}_t (u) = (id_A + t (L_x - R_x )) \circ (\sum_{i \geq 0} t^i T_i ) \circ (id_M - t (l_x- r_x) + t^2 (l_x - r_x)^2 - \cdots ) (u), ~~~ \text{ for } u \in M.
\end{align*}
Hence
\begin{align*}
\overline{T}_t (u) \quad (\text{mod } t^2) =~& (id_A + t (L_x - R_x )) \circ (T + t T_1) (u - t(xu-ux))  \quad (\text{mod } t^2)\\
=~& (id_A + t (L_x- R_x )) (T(u) + t T_1 (u) - t T(xu -ux))  \quad (\text{mod } t^2)\\
=~& T(u) + t \big( x \cdot T(u) - T(u) \cdot x -  T(xu - ux) +  T_1 (u) \big).
\end{align*}
The coefficient of $t$ is zero as $T_1 = d_H (x)$. See \eqref{zero-diff} for instance. Therefore, $\overline{T}_t$ is of the form $\overline{T}_t = T + \sum_{i \geq 2} t^i \overline{T}_i$. By repeating this argument, one get the equivalence between $T_t$ and $T$. Hence the proof.
\end{proof}

\subsection{Order $n$ deformation of an $\O$-operator}
In this section, We introduce a cohomology
class associated to any order $n$ deformation of an $\O$-operator. We prove that
an order $n$ deformation of an $\O$-operator is extensible if and only if this cohomology
class is trivial. This cohomology class  is called the obstruction class of an order $n$
deformation being extensible.
Let $(A,\mu,\a)$ be a Hom-associative algebra with a representation $(M,l,r,\phi)$.
Let $T: M \to A$ be an $\O$-operator on $A$ with respect to $M$.
An order $n$ deformation  of the $\O$-operator $T$ is given by a $\mathbb{K}[[t]]/(t^{n+1})$-linear map
$$T_t =T+ \sum_{i=1}^n t^i T_i,\quad T_i \in C^1_{\phi,\a}(M,A)$$
such that
\begin{align*}
    T_t(u) \cdot T_t(v)=T_t(T_t(u)v+uT_t(v))\quad \textrm{mod}(t^{n+1}),\ \forall u,v \in M.
\end{align*}
Which is equivalent to the following system of equations
\begin{align*}
   \sum_{i+j=k} T_i(u) \cdot T_j(v)=\sum_{i+j=k} T_i(T_j(u)v+uT_j(v)),\quad \textrm{for any}\ k=0,1,\cdots n.
\end{align*}
or equivalently,
\begin{align}
    \llbracket T,T_k \rrbracket_\a=-\frac{1}{2} \sum_{i+j=k,i,j \geqslant 1}   \llbracket T_i,T_j \rrbracket_\a,\quad \textrm{for any}\ k=0,1,\cdots n.
\end{align}
\begin{defn}
Let $T_t =T+ \sum_{i=1}^n t^i T_i$ be an  order $n$ deformation of an $\O$-operator $T$ on $A$ with respect to $M$. If there exists  a $1$-cochain $T_{n+1} \in C^1_{\phi,\a}(M,A)$ such that $\widetilde{T_t}=T_t+t^{n+1}T_{n+1}$ is an order $(n+1)$ deformation of  $T$, then we say  that $T_t$ extends to a deformation of order $(n+1)$.
In this case, we say that $T_t$ is extendable.
\end{defn}

Let us observe that the map $\widetilde{T}_t:=T_t+t^{n+1}T_{n+1}$ is an extension of the order $n$ deformation $T_t$ if and only if
$$\sum_{\substack{i+j=n+1\\i,j\geq 0}} \llbracket T_i,T_j \rrbracket_\a=0.$$
That is
\begin{align}
\llbracket T,T_{n+1} \rrbracket_\a=-\frac{1}{2} \sum_{\substack{i+j=n+1\\ i,j> 0}}\llbracket T_i,T_j\rrbracket_\a.
\end{align}

\begin{defn}
Let $T_t$ be an order $n$ deformation of the $\mathcal{O}$-operator $T$. Let us consider a $2$-cochain $\Theta_T \in C^2_{\phi,\alpha}(M, A)$ defined as follows
\begin{equation}\label{Obst}
\Theta_T =-\frac{1}{2}\sum_{i+j=n+1;~i,j>0}\llbracket T_i,T_j \rrbracket_\a.
\end{equation}
The $2$-cochain $\Theta_T$ is called the obstruction cochain for extending the deformation $T_t$ of order $n$ to a deformation of order $n+1$.
From equation \eqref{Obst} and using graded Jacobi identity of the bracket $\llbracket \cdot,\cdot \rrbracket_\a$, it follows that $\Theta_T$ is a $2$-cocycle.
\end{defn}
The cohomology class $|\Theta_T| \in H^2(M,A)$  is called the obstruction class of $T_t$  being extendable.

\begin{thm}
An order $n$ deformation $T_t$ extends to a deformation of next order if and only if the obstruction class $[\Theta_T]$ is trivial.
\end{thm}

\begin{proof}
Suppose that a deformation of $\mathcal{O}$-operator $T_t$ of order $n$ extends to a deformation of order $n + 1$. Then
$$\sum_{i+j=n+1}( T_i(u) \cdot T_j(v)- T_i(T_j(u)v+uT_j(v)))= 0.$$
As a result, we get $\Theta_T =-\frac{1}{2}\sum_{i+j=n+1;~i,j>0}\llbracket T_i,T_j \rrbracket_\a.$ So, the obstruction class vanishes.
Conversely,  if the obstruction class $|\Theta_T|$ is trivial, suppose that $\Theta_T= -d_T \mathcal{T}_{n+1}$  for some $1$-cochain $\mathcal{T}_{n+1}$. 
Put $\tilde{T}_t=T_t+\mathcal{T}t^{n+1}$. Then $\tilde{T}$ is a deformation of order $n+1$, which means that $T_t$ is extendable.
\end{proof}

\begin{corollary}
If $H^2 (M, A) = 0$,  then every finite order deformation of $T$ extends to a deformation of next order.
\end{corollary}

\begin{corollary}
If $H^2 (M, A) = 0$,  then every $1$-cocycle in $Z^1 (M, A)$ is the infinitesimal of some formal deformation of the $\mathcal{O}$-operator $T$.
\end{corollary}



\begin{thebibliography}{BFGM03}
\bibitem{baxter}
G. Baxter, An analytic problem whose solution follows from a simple algebraic identity, {\em Pacific J. Math.} 10 (1960) 731-742.

\bibitem{rota}
G.-C. Rota, Baxter algebras and combinatorial identities I, II, {\em Bull. Amer. Math. Soc.} 75 (1969) 325–329, 330–334.

\bibitem{O-op on hom-lie}
S. K. Mishra, and A. Naolekar, $\mathcal {O} $-Operators on Hom-Lie algebras. Preprint arXiv:2007.09440 (2020).



\bibitem{amm-ej-makh}
F. Ammar, Z. Ejbehi and A. Makhlouf, Cohomology and deformations of Hom-algebras,
{\em J. Lie Theory} 21 (2011), no. 4, 813-836.

\bibitem{aguiar2}
M. Aguiar, Infinitesimal Hopf algebras, {\em New trends in Hopf algebra theory (La falda, 1999)}, 1-29, Contemp. Math., 267, {\em Amer. Math. Soc., Providence, RI}, 2000.

\bibitem{aguiar}
M. Aguiar, Pre-Poisson algebras, {\em Lett. Math. Phys.} 54 (2000), no. 4, 263-277.

\bibitem{aguiar3}
M. Aguiar, On the associative analog of Lie bialgebras, {\em J. Algebra} 244 (2001), no. 2, 492-532.




\bibitem{bai1}
C. Bai, Double constructions of Frobenius algebras, Connes cocycles and their duality, {\em J. Noncommut. Geom.,} 4 (2010), pp. 475-530

\bibitem{bala}
D. Balavoine, Deformations of algebras over a quadratic operad, {\em  Operads: Proceedings of Renaissance Conferences (Hartford, CT/Luminy, 1995),} 207-234,
Contemp. Math., 202, {\em Amer. Math. Soc., Providence, RI,} 1997.

\bibitem{balav}
D. Balavoine, D\'{e}formations des alg\'{e}bres de Leibniz (in French),
{\em C. R. Acad. Sci. Paris S\'{e}r. I Math.} 319 (1994), no. 8, 783-788.


\bibitem{bhas}
K. H. Bhaskara and K. Viswanath, Calculus on Poisson manifolds,
{\em Bull. London Math. Soc.} 20 (1988), no. 1, 68-72.

\bibitem{grab}
J. F. Cari\~{n}ena, J. Grabowski and G. Marmo, Quantum Bi-Hamiltonian Systems, {\em Internat. J. Modern. Phys. A} 15 (2000) 4797-4810.

\bibitem{CaiSheng} L. Cai and  Y. Sheng,   Hom-big brackets: theory and applications, \emph{SIGMA} 12 (2016), 014, 18 pages.

\bibitem{cart}
P. Cartier, On the structure of free Baxter algebras, {\em Advances in Math.} 9 (1972), 253-265.

\bibitem{conn}
A. Connes and D. Kreimer, Renormalization in quantum field theory and the Riemann-Hilbert problem. I. The Hopf algebra structure of graphs and the main theorem,
{\em Comm. Math. Phys.} 210 (2000), no. 1, 249-273.

\bibitem{courant}
T. J. Courant, Dirac manifolds, {\em Trans. Amer. Math. Soc.} 319 (1990), no. 2, 631-661.

\bibitem{das}
A. Das, Cohomology and deformations of dendriform algebras, and $\text{Dend}_\infty$-algebras, preprint, arXiv:1903.11802

\bibitem{Gers-Hom-ass}
A. Das,  Gerstenhaber algebra structure on the cohomology of a hom-associative algebra. {\em  Proc. Indian Acad. Sci. Math. Sci.} 130 (2020), no. 1, Paper No. 20, 20 pp. 

\bibitem{cohom hom dend}
A. Das, Cohomology and deformations of Hom-dendriform algebras and coalgebras.   Colloq. Math. 163 (2021), no. 1, 37–52.

\bibitem{def O-op on ass al}
A. Das, Deformations of associative Rota-Baxter operators.  {\em  J. Algebra } 560 (2020), 144--180. 

\bibitem{dz}
A. Dzhumadil'daev, Cohomologies and deformations of right-symmetric algebras,
{\em J. Math. Sci. (New York)} 93 (1999), no. 6, 836-876.

\bibitem{fard-guo}
K. Ebrahimi-Fard and L. Guo, Rota-Baxter algebras and dendriform algebras,
{\em J. Pure Appl. Algebra} 212 (2008), no. 2, 320-339.

\bibitem{fro-nij}
A. Fr\"{o}licher and A. Nijenhuis, A theorem on stability of complex structures,
{\em Proc. Nat. Acad. Sci. U.S.A.} 43 (1957), 239-241.

\bibitem{gers}
M. Gerstenhaber, On the deformation of rings and algebras,
{\em Ann. of Math.} (2) 79 (1964), 59-103.

\bibitem{gers2}
M. Gerstenhaber, The cohomology structure of an associative ring,
{\em Ann. of Math.} (2) 78 (1963), 267-288.

\bibitem{guo-book}
L. Guo, An introduction to Rota-Baxter algebra,
Surveys of Modern Mathematics, 4. {\em International Press, Somerville, MA; Higher Education Press, Beijing,} 2012.

\bibitem{guo-adv}
L. Guo and W. Keigher, Baxter algebras and shuffle products, {\em Adv. Math.} 150 (2000), no. 1, 117-149.

\bibitem{hls}
J.T. Hartwig, D. Larsson, and S.D. Silvestrov, Deformations of Lie algebras using $\sigma$-derivations, {\em J. Algebra} 295 (2006) 314--361.

\bibitem{hoch}
G. Hochschild, On the cohomology groups of an associative algebra,
{\em Ann. of Math.} (2) 46 (1945), 58-67.

\bibitem{hurleMakhlouf} B. Hurle and A. Makhlouf, $\alpha$-type Hochschild cohomology of Hom-associative algebras and bialgebras. {\em J. Korean Math. Soc.} 56 (2019), no. 6, 1655–1687

\bibitem{kns}
K. Kodaira, L. Nirenberg and D. C. Spencer,
On the existence of deformations of complex analytic structures,
{\em Ann. of Math. (2)} 68 (1958), 450â€“459.

\bibitem{kod-sp} K. Kodaira and D. C. Spencer, On deformations of complex analytic structures. I, II,
{\em Ann. of Math. (2)} 67 (1958), 328-466.

\bibitem{kont} M. Kontsevich, Deformation quantization of Poisson manifolds,
{\em Lett. Math. Phys.} 66 (2003), no. 3, 157-216.

\bibitem{kelley}
J. L. Kelley, Averging operators on $C_\infty (X)$, {\em Illinois J. Math.} 2 (1958), 214-223

\bibitem{LazarevSheng} A. Lazarev, Y. Sheng and R. Tang, Deformations and homotopy theory  of relative Rota-Baxter Lie algebras, {Comm. Math. Phys.} 383 (2021),  595--631.

\bibitem{loday}
J.-L. Loday, Dialgebras, {\em  Dialgebras and related operads,} 7-66,
Lecture Notes in Math., 1763, {\em Springer, Berlin,} 2001.

\bibitem{nij-ric}
A. Nijenhuis and R. W. Richardson, Cohomology and deformations in graded Lie algebras,
{\em Bull. Amer. Math. Soc.} 72 (1966), 1-29.

\bibitem{pei-guo} J. Pei and L. Guo, Averaging algebras, Schr\"{o}der numbers, rooted trees and operads,  {\em J. Algebraic Combin.} 42 (2015), no. 1, 73-109.

\bibitem{rota2}
G.-C. Rota, On the representation of averaging operator, {\em Rendiconti del Seminario Matematico della
Universita di Padova} 30 (1960), 52-64.


\bibitem{rey}
O. Reynolds, On the dynamic theory of incompressible viscous fluids, {\em Phil. Trans. Roy. Soc.} A 136
(1895), 123-164

\bibitem{sheng}
J. Liu, C. Bai and Y. Sheng, Compatible $\mathcal{O}$-operators on bimodules over associative algebras, {\em J. Algebra} 532 (2019), 80-118.



\bibitem{ms}
A. Makhlouf and S. Silvestrov, Hom-algebra structures, J. Gen. Lie Theory Appl. 2 (2008) 51--64.

\bibitem{hom dend}
A. Makhlouf. Hom-dendriform algebras and Rota-Baxter Hom-algebras. In : Operads And Universal Algebra.  p. 147--171.
{\em Nankai Ser. Pure Appl. Math. Theoret. Phys., 9, World Sci. Publ., Hackensack,} NJ, 2012.
\bibitem{makh-sil} A. Makhlouf and S. Silvestrov, Notes on $1$-parameter formal deformations of Hom-associative and Hom-Lie algebras,
{\em Forum Math.} 22 (2010), no. 4, 715-739.

\bibitem{sheng}
 Y. Sheng, Representations of hom-Lie algebras. {\em Algebr. Represent. Theory} 15 (2012), no. 6, 1081–1098.

\bibitem{tang}
R. Tang, C. Bai, L. Guo and Y. Sheng, Deformations and their controlling cohomologies of $\mathcal{O}$-operators,
{\em Comm. Math. Phys.} 368 (2019), no. 2, 665-700.



\bibitem{uchino1}
K. Uchino, Quantum analogy of Poisson geometry, related dendriform algebras and O-operators, {\em Lett. Math. Phys.} 85 (2008) 91-109.

\bibitem{uchino2}
K. Uchino, Twisting on associative algebras and Rota-Baxter type operators, {\em J. Noncommut. Geom. } 4(3) (2010) 349-379.




\bibitem{uchino3}
K. Uchino, Derived bracket construction and Manin products, {\em Lett. Math. Phys.} 93 (2010), no. 1, 37-53.

\bibitem{voro}
Th. Voronov, Higher derived brackets and homotopy algebras,
{\em J. Pure Appl. Algebra} 202 (2005), no. 1-3, 133-153.

\bibitem{wang} Q. Wang, Y. Sheng, C. Bai and J. Liu, Nijenhuis operators on pre-Lie algebras, {\em Commun. Contemp. Math.} 21 (2019), no. 7, 1850050, 37 pp.




\end{thebibliography}
\end{document}